\documentclass[11pt]{article}

\usepackage{amsmath,amssymb,amsthm}
\usepackage{tikz}
\usepackage[margin=1.1in]{geometry}
\usepackage[colorlinks=true,linkcolor=blue,citecolor=blue,urlcolor=blue]{hyperref}
\usepackage{comment}
\usepackage{thm-restate}
\newtheorem{theorem}{Theorem}[section]
\newtheorem{lemma}[theorem]{Lemma}
\newtheorem{observation}[theorem]{Observation}

\newtheorem{conjecture}[theorem]{Conjecture}
\newtheorem{remark}[theorem]{Remark}

\newcommand{\diameter}{\operatorname{diam}}
\newcommand{\ordiam}
{\overrightarrow{\operatorname{diam}}}
\newcommand{\Int}{\operatorname{Int}}
\newcommand{\Ext}{\operatorname{Ext}}

\makeatletter
               {\list{}{\leftmargin=0pt % <------- Adjust this length
                        \labelwidth\z@ \itemindent-\leftmargin
                        }}%
               {\endlist}
\makeatother

\newcommand{\diam}{\operatorname{diam}}

\title{Improved bounds on the oriented diameter of planar triangulations}

\author{Xiaonan Liu \thanks{Department of Mathematics, Louisiana State University, Baton Rouge, LA, 70803
({\tt xliu20@lsu.edu}).}
}
\begin{document}

\maketitle

\begin{abstract}
The {oriented diameter} of a connected bridgeless graph $G$, denoted by
$\ordiam(G)$, is the minimum diameter among all strong orientations of
$G$. We study the oriented diameter of planar triangulations, and show that $\ordiam(G)\leq \frac{2n+44}{5}$ for any $n$-vertex planar triangulation $G$. This improves the leading constant in the previous best general upper bound $\lceil \frac{n}{2}\rceil$, due to Ge, Liu, and Wang, from $1/2$ to $2/5$. We also prove that every $n$-vertex $4$-connected planar triangulation satisfies $\ordiam(G)\leq \frac{n+17}{3}$.
\end{abstract}

\section{Introduction}\label{sec:introduction}

An \emph{orientation} of an undirected graph $G$ is obtained by assigning one of the two possible directions to every edge of $G$. An orientation is \emph{strong} if the resulting digraph is strongly connected. By a classical theorem of Robbins~\cite{Robbins1939}, a connected graph admits a strong orientation if and only if it is bridgeless. For a connected bridgeless graph $G$, the \emph{oriented diameter} of $G$ is
\[
    \ordiam(G)
    :=\min\{\diameter(D):D\text{ is a strong orientation of }G\}.
\]
The oriented-diameter problem asks how efficiently the edges of a bidirectional network can be directed while preserving short routes between every ordered pair of vertices. It is therefore naturally
motivated by network design and routing. For example, when two-way
streets are converted into one-way streets, global accessibility should
be preserved without introducing excessively long detours. For planar graphs, the problem also interacts with several classical themes in structural graph theory, including connectivity, planar duality, separating cycles, and recursive decompositions.

The systematic study of oriented diameter began with Chv\'atal and Thomassen~\cite{ChvatalThomassen1978}. They proved that deciding whether a graph admits an orientation of diameter at most two is NP-complete. They also showed that every bridgeless graph of diameter $d$ has oriented diameter at most $2d^2+2d$, and constructed examples for which every strong orientation has diameter at least $\frac12d^2+d$. The general upper bound was subsequently improved by Babu, Benson, Rajendraprasad and Vaka~\cite{BBRV2021}. Oriented diameter has also been studied in terms of the domination number, the minimum and maximum degrees, and other structural parameters; see, for example,~\cite{FMPR2004,BauDankelmann2015,Surmacs2017,DGS2018,CDS2019,Cochran2024,CochranWang2026} and the survey of Koh and Tay~\cite{KohTay2002}. For graphs of diameter three, Wang and Chen~\cite{WangChen2022} determined the optimal universal bound to be nine.

In this paper, we study the problem for planar triangulations. A \emph{planar triangulation} is a planar graph that admits a plane embedding in which every face is bounded by a triangle. A \emph{plane triangulation} is a planar triangulation together with a fixed such embedding in the plane. 
 From an algorithmic perspective, Eggemann and Noble~\cite{EggemannNoble2012} proved that, for every fixed positive integer $k$, one can decide in linear time whether a planar graph admits an orientation of diameter at most $k$.

The first general bound specifically for planar triangulations was obtained by Mondal, Parthiban, and Rajasingh~\cite{MPR2024}. They determined the oriented diameter of triangular grid graphs, proved that every $n$-vertex planar triangulation has oriented diameter at most
$\frac n2+O(\sqrt n)$, and constructed, for every $n$ divisible by three, an $n$-vertex planar triangulation with oriented diameter at least $n/3$. Their lower-bound examples are the \emph{nested-triangle triangulations}: the vertex set is partitioned into a sequence of nested triangles, and each annular region between two consecutive triangles is triangulated in a fixed way. In these examples, two vertices on the innermost and outermost triangles lie on no cycle of length less than $2n/3$, forcing one of the two directed distances between them to be at least $n/3$ in every strong orientation.

The upper bound was later improved by Ge, Liu and Wang~\cite{GLW2025}, who proved that every $n$-vertex $2$-connected near triangulation has oriented diameter at most $\lceil \frac n2\rceil$, apart from seven explicitly described small exceptions. Their result is best possible for near triangulations because maximal outerplanar graphs may have oriented diameter $\lceil \frac n2\rceil$; the sharp maximal-outerplanar bound was established by Wang, Chen, Dankelmann, Guo, Surmacs, and Volkmann~\cite{WCDGSV2021}. In particular, apart from finitely many small graphs, the best known universal upper bound for planar triangulations is $\lceil \frac n2\rceil$.

Stronger bounds are known under additional connectivity assumptions. Wang~\cite{Wang2025} studied both the oriented diameter and the oriented radius of triangulations. Among other results, Wang~\cite{Wang2025} proved that every $5$-connected $n$-vertex planar triangulation has oriented diameter at most $\frac{2n}{5}+O(1)$. 

For arbitrary planar triangulations, however, the principal gap remained between the lower bound $n/3$ provided by nested triangles and the general upper bound $n/2$. These results naturally suggest the following
conjecture.
\begin{conjecture}\label{conj:main}
There exists an absolute constant $C$ such that every $n$-vertex planar triangulation $G$ satisfies
\[
    \ordiam(G)\le \frac n3+C.
\]
\end{conjecture}

Conjecture~\ref{conj:main}, if true, would be asymptotically best possible. Although we do not resolve the conjecture in this paper, we improve the general upper bound from $n/2+O(1)$ to $2n/5+O(1)$. We also prove Conjecture~\ref{conj:main} for $4$-connected planar triangulations. Both results are obtained by combining two complementary estimates: one is effective when the ordinary diameter is small, while the other becomes stronger as the ordinary diameter increases.

Motivated by the general relationship between diameter and oriented diameter studied by Chv\'atal and Thomassen~\cite{ChvatalThomassen1978}, we first establish a linear bound for planar triangulations in terms of their ordinary diameter.

\begin{restatable}{theorem}{diambound}\label{thm:ordiam_vs_diam}
Let $G$ be a planar triangulation such that $G\neq K_4$. Then $G$ admits a strong orientation $D$ such that $d_D(u,v)\leq 2d_G(u,v)$ for any $u,v\in V(G)$. In particular, we have that \[\ordiam(G)\leq 2\diam(G).\]
\end{restatable}

Theorem~\ref{thm:ordiam_vs_diam} gives a strong estimate when $\diameter(G)$ is small. Note that it follows from Menger's theorem that every $n$-vertex $k$-connected  graph has diameter at most $\lfloor \frac{n+k-2}{k}\rfloor$. Hence, Theorem~\ref{thm:ordiam_vs_diam} yields the following bounds for planar triangulations of different connectivities.
\begin{theorem}
 For $k\in \{3, 4, 5\}$, every $n$-vertex $k$-connected planar triangulation  $G$ has 
 \[\ordiam(G)\leq 2\lfloor \frac{n+k-2}{k}\rfloor.\]
\end{theorem}
Note that, for planar triangulations with connectivity $3$, the above bound is weaker than the general bound  $\lceil \frac n2\rceil$ due to Ge, Liu, and Wang, while in the connectivity-$4$ case, it is asymptotically the same as $\lceil \frac n2\rceil$. However, for $5$-connected planar triangulations, it gives the bound $2\lfloor \frac{n+3}{5}\rfloor$, which is asymptotically the same as the bound obtained by Wang~\cite{Wang2025}.

The following result handles the case in which $G$ has large diameter.

\begin{restatable}{theorem}{longdiam}\label{thm:long_diameter}
  Let $G$ be an $n$-vertex planar triangulation. Then \[\ordiam(G)\leq \frac{n-\diam(G)}{2}+11.\]  
\end{restatable}
Balancing the two bounds in Theorems~\ref{thm:ordiam_vs_diam} and~\ref{thm:long_diameter} gives the main result of this paper. Since the proof is short, we give it here.
\begin{theorem}\label{main}
Let $G$ be an $n$-vertex planar triangulation. Then \[\ordiam(G)\leq \frac{2n+44}{5}.\]
\end{theorem}
\begin{proof}
Let $G$ be an $n$-vertex planar triangulation. Since $\ordiam(K_4)=3< \frac{2\cdot4+44}{5}$, we may assume that $G\neq K_4$. By Theorems~\ref{thm:ordiam_vs_diam} and \ref{thm:long_diameter}, we have that \[\ordiam(G)\leq \min \{ 2\diam(G), {\frac{n-\diam(G)}{2}+11}\}\leq \frac{2n+44}{5}.\]
\end{proof}

For $4$-connected plane triangulations, we obtain a stronger large-diameter estimate.
\begin{restatable}{theorem}{highconn}\label{main:4-conn}
 Let $G$ be an $n$-vertex $4$-connected planar triangulation. Then \[\ordiam(G)\leq \frac{n-2\diam(G)}{2}+\frac{17}{2}.\]   
\end{restatable}
Combining Theorem~\ref{main:4-conn} with Theorem~\ref{thm:ordiam_vs_diam} establishes Conjecture~\ref{conj:main} for $4$-connected plane triangulations.
\begin{theorem}
  Let $G$ be an $n$-vertex $4$-connected planar triangulation. Then \[\ordiam(G)\leq \frac{n+17}{3}.\]     
\end{theorem}

\medskip
\noindent
\textbf{Organization.}
The remainder of the paper is organized as follows. In Section~\ref{sec:preliminaries}, we collect additional definitions and known results, and prove Theorem~\ref{thm:ordiam_vs_diam}. We prove Theorem~\ref{thm:long_diameter} in Section~\ref{sec:long_diameter}, and prove Theorem~\ref{main:4-conn} in Section~\ref{sec:4-conn}.

\medskip

We conclude this section with some terminology and notation. For any positive integer $k$, let $[k]:=\{1,2, \ldots,k\}$. 

Let $G$ and $H$ be graphs. We use $G\cup H$ and $G\cap H$ to denote the union and intersection of $G$ and $H$, respectively. For any $S\subseteq V(G)$, we use $G[S]$ to denote the subgraph of $G$ induced by $S$, and let $G-S = G[V(G)\setminus S]$. For a subgraph $T$ of $G$, we often write $G-T$ for $G-V(T)$ and write $G[T]$ for $G[V(T)]$.  A path (respectively, cycle) is often represented as a sequence (respectively, cyclic sequence) of vertices, with consecutive vertices 
being adjacent.  Given a path $P$ and distinct vertices $x,y\in V(P)$, we use $xPy$ to denote the subpath of $P$ between $x$ and $y$.

Let $G$ be a plane graph. The {\it outer walk} of $G$ consists of vertices and edges of $G$ incident with the outer face of $G$. If the outer walk is a cycle in $G$, we call it  {\it outer cycle} instead. For a cycle $C$ in  $G$, we use $\overline{\Int_G(C)}$ and  $\overline{\Ext_G(C)}$ to denote the subgraphs of $G$ consisting of all vertices and edges contained in the bounded and unbounded closed regions of the plane determined by $C$, respectively. The {\it interior} of $C$ is $\Int_G(C):=\overline{\Int_G(C)}-C$, and the \emph{exterior} of $C$ is $\Ext_G(C):=\overline{\Ext_G(C)}-C=G-\overline{\Int_G(C)}$. When $G$ is clear from the context, we omit the subscript $G$. For convenience, we may also write $\overline{C}=\overline{\Int(C)}$ . %For any distinct vertices $u,v\in V(C)$, we use $uCv$ to denote the subpath of $C$ from $u$ to $v$ in clockwise order. 

\section{Preliminaries}\label{sec:preliminaries}

All graphs in this paper are finite and simple. A directed graph $D = (V(D),A(D))$ is a graph with a vertex set $V(D)$ and an edge set $A(D)$ consisting of ordered pairs of vertices, called \textit{arcs} or directed edges. We use $uv$ or $u\to v$ to denote the arc $(u,v)$, i.e., the arc oriented from $u$ to $v$. Given an undirected graph $G=(V(G),E(G))$, an \textit{orientation} of $G$ is a directed graph such that each edge in $E(G)$ is assigned a direction. Given a (directed) graph $D$ and two vertices $u,v \in V(D)$, the \textit{distance} from $u$ to $v$ in $D$, denoted by $d_D(u,v)$, is the number of edges of a shortest (directed) path from $u$ to $v$ in $D$; if no such path exists, set $d_D(u,v)=\infty$.  More generally, for nonempty vertex sets $X,Y\subseteq V(D)$, define
\[
    d_D(X,Y):=\min\{d_D(x,y):x\in X,\ y\in Y\}.
\]
When one of the sets consists of a single vertex, we omit the braces. The \emph{diameter} of a (directed) graph $D$ is
\[
    \diameter(D):=\max_{u,v\in V(D)}d_D(u,v).
\]

An edge $e \in E(G)$ is called a \textit{bridge} if $G-e$ is disconnected. A graph is called \textit{bridgeless} if it contains no bridge. A directed graph $D$ is called \textit{strongly connected} if for any two vertices $u,v \in V(D)$, there exists a directed path from $u$ to $v$. An orientation of a graph is \emph{strong} if it is strongly connected. Recall that Robbins~\cite{Robbins1939} showed in 1939 that every connected bridgeless graph has a strong orientation. The oriented diameter of a connected bridgeless graph $G$, denoted by $\ordiam(G)$, is defined as
$$\ordiam(G)=\min\{\diam(D):\textrm{$D$ is a strong orientation of $G$}\}.$$ An orientation $D$ of $G$ is called \textit{optimal} if $\ordiam(G)=\diam(D)$. 

\medskip
A \textit{near triangulation} is a plane graph such that every face except possibly the outer face is bounded by a triangle. In this paper, we call a graph $G$ a {near triangulation with outer cycle $C$} if $G$ admits a plane embedding in which the outer face is bounded by $C$ and every other face is bounded by a triangle. Let $K_4^{-}$ denote the graph obtained from $K_4$ by deleting an arbitrary edge. Let $W_5$ be the wheel graph on six vertices obtained by connecting every vertex of a $C_5$ to a new vertex. Let $G_6^1, G_6^2, G_6^3, G_8^1$ be the graphs in Figure \ref{fig:exceptions}.
 \begin{figure}[htb]
    \hbox to \hsize{
	\hfil
	\resizebox{2.8cm}{!}{\begin{tikzpicture}[
    scale=1,
    vertex/.style={circle, draw=black, fill=black, scale=0.5},
    redvertex/.style={circle, draw=red, fill=red, scale=0.5},
    arc/.style={->, >=stealth, line width=0.9pt}
]

\node[vertex]    (e1) at (180:1.4) {};
\node[vertex] (e2) at (120:1.4) {};
\node[vertex] (e3) at (60:1.4) {};
\node[vertex] (e4) at (0:1.4) {};
\node[vertex] (e5) at (-60:1.4) {};
\node[vertex]    (e6) at (-120:1.4) {};
\draw (e6) -- (e1);
\draw (e1) -- (e2);
\draw (e2) -- (e3);
\draw(e3) -- (e4);
\draw (e4) -- (e5);
\draw (e5) -- (e6);
\draw (e6) -- (e2);
\draw (e3) -- (e6);
\draw (e5) -- (e3);
\end{tikzpicture}}%
	\hfil
	\resizebox{2.8cm}{!}{\begin{tikzpicture}[
    scale=1,
    vertex/.style={circle, draw=black, fill=black, scale=0.5},
    redvertex/.style={circle, draw=red, fill=red, scale=0.5},
    arc/.style={->, >=stealth, line width=0.9pt}
]

\node[vertex]    (e1) at (180:1.4) {};
\node[vertex] (e2) at (120:1.4) {};
\node[vertex] (e3) at (60:1.4) {};
\node[vertex] (e4) at (0:1.4) {};
\node[vertex] (e5) at (-60:1.4) {};
\node[vertex]    (e6) at (-120:1.4) {};
\draw (e6) -- (e1);
\draw (e1) -- (e2);
\draw (e2) -- (e3);
\draw(e3) -- (e4);
\draw (e4) -- (e5);
\draw (e5) -- (e6);
\draw (e6) -- (e2);
\draw (e3) -- (e6);
\draw (e6) -- (e4);
\end{tikzpicture}}%
	\hfil
    \resizebox{3.5cm}{!}{\begin{tikzpicture}[
    scale=1,
    vertex/.style={circle, draw=black, fill=black, scale=0.65},
    redvertex/.style={circle, draw=red, fill=red, scale=0.5},
    arc/.style={->, >=stealth, line width=0.9pt}
]

\node [vertex] (v1) at (150:2) {};
\node [vertex] (v2) at (0,0) {};
\node [vertex] (v3) at (30:2) {};
\node [vertex] (v4) at (3,-0.5) {};
\node [vertex] (v5) at (1.73,-2) {};
\node [vertex] (v6) at (-90:2) {};

\draw (v1) -- (v2);
\draw (v2) -- (v3);
\draw (v3) -- (v4);
\draw (v4) -- (v5);
\draw (v5) -- (v6);
\draw (v6) -- (v1);
\draw (v2) -- (v6);
\draw (v3) -- (v6);
\draw (v3) -- (v5);
\draw (v1) -- (v3);

\end{tikzpicture}}%
    \hfil
    \resizebox{3.35cm}{!}{\begin{tikzpicture}[
    scale=1,
    vertex/.style={circle, draw=black, fill=black, scale=0.45},
    redvertex/.style={circle, draw=red, fill=red, scale=0.5},
    arc/.style={->, >=stealth, line width=0.9pt}
]

\node [vertex] (v1) at (-1,1) {};
\node [vertex] (v2) at (0,1) {};
\node [vertex] (v3) at (1,1) {};
\node [vertex] (v4) at (2,1) {};
\node [vertex] (v5) at (2,-1) {};
\node [vertex] (v6) at (1,-1) {};
\node [vertex] (v7) at (0,-1) {};
\node [vertex] (v8) at (-1,-1) {};

\draw (v1) -- (v2);
\draw (v2) -- (v3);
\draw (v3) -- (v4);
\draw (v4) -- (v5);
\draw (v5) -- (v6);
\draw (v6) -- (v7);
\draw (v7) -- (v8);
\draw (v8) -- (v1);
\draw (v2) -- (v7);
\draw (v3) -- (v6);
\draw (v1) -- (v7);
\draw (v3) -- (v7);
\draw (v3) -- (v5);

\end{tikzpicture}}%
    \hfil
    }
    \caption{Some small exceptions to Theorem \ref{thm:near_triangulation_diameter}: $G_6^1, G_6^2, G_6^3, G_8^1$ (from left to right).}
    \label{fig:exceptions}
    \end{figure}
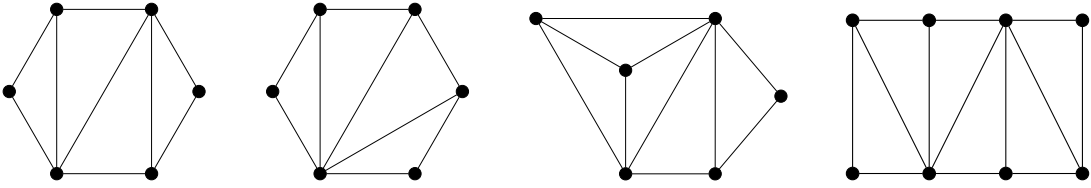

 We first state the result of Ge, Liu, and Wang \cite{GLW2025} for $2$-connected near triangulations, which we use to prove Theorems~\ref{thm:long_diameter} and \ref{main:4-conn}. This result is also the best known asymptotic bound for general planar triangulations.
\begin{theorem}\cite{GLW2025}\label{thm:near_triangulation_diameter}
    Let $G$ be an $n$-vertex $2$-connected near triangulation that is not one of the graphs in $\{K_4^-, K_4, W_5, G_6^1, G_6^2, G_6^3, G_8^1\}$. Then 
    $\ordiam(G)\leq \lceil{\frac{n}{2}}\rceil$.
\end{theorem}
 The following lemma is a consequence of \cite{GLW2025}. We give the proof here for completeness.
\begin{lemma}\cite{GLW2025}\label{lem:special_orientation}
    Let $G$ be an $n$-vertex $2$-connected  near triangulation with outer cycle $C$. For any vertex $v\in V(C)$, $G$ admits an orientation $D$ such that $\diam(D)=\ordiam(G)\leq \frac{n}{2}+1$ and $\max\{ d_D(v,x), d_D(x, v)\}\leq \frac{n+1}{2}$ for any $x\in V(G)$.
    \end{lemma}
\begin{proof}
    Let $G$ be an $n$-vertex $2$-connected  near triangulation with outer cycle $C$. If $G$ is not a graph in $\{K_4^-, K_4, W_5, G_6^1, G_6^2, G_6^3, G_8^1\}$, the conclusion holds for $G$ by Theorem~\ref{thm:near_triangulation_diameter}. Now we suppose $G\in \{K_4^-, K_4, W_5, G_6^1, G_6^2, G_6^3, G_8^1\}$. Note that  $K_4^-$ is a spanning subgraph of $K_4$ and $G_6^2$ is a spanning subgraph of $W_5$ and $G_6^3$. Hence, it suffices to give explicit optimal orientations for graphs in $\{K_4^-, G_6^1, G_6^2, G_8^1\}$ (see Figure~\ref{fig:exceptions_orientation}). In each orientation in Figure~\ref{fig:exceptions_orientation}, the vertices $v$ for which Lemma~\ref{lem:special_orientation} holds are labeled as red vertices.  This proves Lemma~\ref{lem:special_orientation}.
\end{proof}
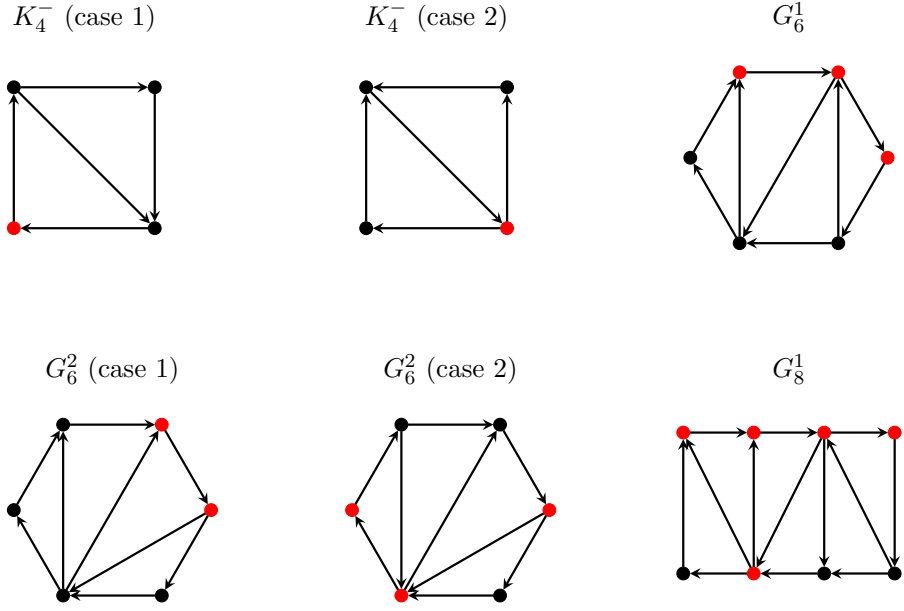
\begin{figure}[htbp]
\label{fig:exceptions_orientation}
\centering
\resizebox{12cm}{!}{\begin{tikzpicture}[
    scale=1,
    vertex/.style={circle, draw=black, fill=black, scale=0.5},
    redvertex/.style={circle, draw=red, fill=red, scale=0.5},
    arc/.style={->, >=stealth, line width=0.9pt}
]

% =========================================================
% Top row
% =========================================================

% K4^- (case 1)
\begin{scope}[shift={(0,0)}]
\node at (1,3) {$K_4^{-}$ (case 1)};
\node[redvertex]    (a1) at (0,0) {};
\node[vertex]    (a2) at (0,2) {};
\node[vertex] (a3) at (2,2) {};
\node[vertex]    (a4) at (2,0) {};
\draw[arc] (a1) -- (a2);
\draw[arc] (a2) -- (a3);
\draw[arc] (a3) -- (a4);
\draw[arc] (a4) -- (a1);
\draw[arc] (a2) -- (a4);
\end{scope}

% K4^- (case 2)
\begin{scope}[shift={(5,0)}]
\node at (1,3) {$K_4^{-}$ (case 2)};
\node[vertex]    (b1) at (0,0) {};
\node[vertex]    (b2) at (0,2) {};
\node[vertex]    (b3) at (2,2) {};
\node[redvertex] (b4) at (2,0) {};
\draw[arc] (b1) -- (b2);
\draw[arc] (b3) -- (b2);
\draw[arc] (b4) -- (b3);
\draw[arc] (b4) -- (b1);
\draw[arc] (b2) -- (b4);
\end{scope}

% W5
% \begin{scope}[shift={(11,1)}]
% \node at (0,2) {$W_5$};
% \node[redvertex] (d1) at (90:1.4) {};
% \node[vertex]    (d2) at (18:1.4) {};
% \node[vertex]    (d3) at (-54:1.4) {};
% \node[vertex]    (d4) at (-126:1.4) {};
% \node[vertex]    (d5) at (-198:1.4) {};
% \node[vertex]    (d6) at (0:0) {};
% \draw[arc] (d5) -- (d1);
% \draw[arc] (d2) -- (d1);
% \draw[arc] (d3) -- (d2);
% \draw[arc] (d3) -- (d4);
% \draw[arc] (d4) -- (d5);
% \draw[arc] (d1) -- (d6);
% \draw[arc] (d6) -- (d4);
% \draw[arc] (d6) -- (d3);
% \draw[arc] (d5) -- (d6);
% \draw[arc] (d6) -- (d2);
% \end{scope}

% =========================================================
% Middle row
% =========================================================

% G6^1
\begin{scope}[shift={(11,1)}]
\node at (0,2) {$G_6^{1}$};
\node[vertex]    (e1) at (180:1.4) {};
\node[redvertex] (e2) at (120:1.4) {};
\node[redvertex] (e3) at (60:1.4) {};
\node[redvertex] (e4) at (0:1.4) {};
\node[vertex] (e5) at (-60:1.4) {};
\node[vertex]    (e6) at (-120:1.4) {};
\draw[arc] (e6) -- (e1);
\draw[arc] (e1) -- (e2);
\draw[arc] (e2) -- (e3);
\draw[arc] (e3) -- (e4);
\draw[arc] (e4) -- (e5);
\draw[arc] (e5) -- (e6);
\draw[arc] (e6) -- (e2);
\draw[arc] (e3) -- (e6);
\draw[arc] (e5) -- (e3);
\end{scope}

% G6^2 (case 1)
\begin{scope}[shift={(1.4,-4.0)}]
\node at (0,2) {$G_6^{2}$ (case 1)};
\node[vertex]    (f1) at (180:1.4) {};
\node[vertex]    (f2) at (120:1.4) {};
\node[redvertex] (f3) at (60:1.4) {};
\node[redvertex] (f4) at (0:1.4) {};
\node[vertex]    (f5) at (-60:1.4) {};
\node[vertex]    (f6) at (-120:1.4) {};
\draw[arc] (f6) -- (f1);
\draw[arc] (f1) -- (f2);
\draw[arc] (f2) -- (f3);
\draw[arc] (f3) -- (f4);
\draw[arc] (f4) -- (f5);
\draw[arc] (f5) -- (f6);
\draw[arc] (f6) -- (f2);
\draw[arc] (f6) -- (f3);
\draw[arc] (f4) -- (f6);
\end{scope}

% G6^2 (case 2)
\begin{scope}[shift={(6.2,-4.0)}]
\node at (0,2) {$G_6^{2}$ (case 2)};
\node[redvertex]    (g1) at (180:1.4) {};
\node[vertex] (g2) at (120:1.4) {};
\node[vertex] (g3) at (60:1.4) {};
\node[redvertex]    (g4) at (0:1.4) {};
\node[vertex] (g5) at (-60:1.4) {};
\node[redvertex]    (g6) at (-120:1.4) {};
\draw[arc] (g6) -- (g1);
\draw[arc] (g1) -- (g2);
\draw[arc] (g2) -- (g3);
\draw[arc] (g3) -- (g4);
\draw[arc] (g4) -- (g5);
\draw[arc] (g5) -- (g6);
\draw[arc] (g2) -- (g6);
\draw[arc] (g6) -- (g3);
\draw[arc] (g4) -- (g6);
\end{scope}

% G6^3 (case 1)
% \begin{scope}[shift={(11,-3.75)}]
% \node at (0,1.75) {$G_6^{3}$ (case 1)};

% \node [vertex] (h1) at (150:1.4) {};
% \node [vertex] (h2) at (0,0) {};
% \node [redvertex] (h3) at (30:1.4) {};
% \node [redvertex] (h4) at (2.1,-0.35) {};
% \node [redvertex] (h5) at (1.21,-1.4) {};
% \node [redvertex] (h6) at (-90:1.4) {};

% \draw[arc] (h2) -- (h1);
% \draw[arc] (h2) -- (h3);
% \draw[arc] (h3) -- (h4);
% \draw[arc] (h4) -- (h5);
% \draw[arc] (h5) -- (h6);
% \draw[arc] (h6) -- (h1);
% \draw[arc] (h6) -- (h2);
% \draw[arc] (h6) -- (h3);
% \draw[arc](h3) -- (h5);
% \draw[arc] (h1) -- (h3);
% \end{scope}

% G6^3 (case 2)
% \begin{scope}[shift={(16,-3.75)}]
% \node at (0,1.75) {$G_6^{3}$ (case 2)};

% \node [redvertex] (i1) at (150:1.4) {};
% \node [vertex] (i2) at (0,0) {};
% \node [vertex] (i3) at (30:1.4) {};
% \node [vertex] (i4) at (2.1,-0.35) {};
% \node [vertex] (i5) at (1.21,-1.4) {};
% \node [vertex] (i6) at (-90:1.4) {};

% \draw[arc] (i2) -- (i1);
% \draw[arc] (i2) -- (i3);
% \draw[arc] (i3) -- (i4);
% \draw[arc] (i4) -- (i5);
% \draw[arc] (i5) -- (i6);
% \draw[arc] (i6) -- (i1);
% \draw[arc] (i6) -- (i2);
% \draw[arc] (i3) -- (i6);
% \draw[arc] (i3) -- (i5);
% \draw[arc] (i1) -- (i3);
% \end{scope}
% =========================================================
% Bottom row
% =========================================================

% G8^1
\begin{scope}[shift={(10.5,-3.9)}]
\node at (0.5,1.9) {$G_8^{1}$};
\node [redvertex] (v1) at (-1,1) {};
\node [redvertex] (v2) at (0, 1) {};
\node [redvertex] (v3) at (1,1) {};
\node [redvertex] (v4) at (2, 1) {};
\node [vertex] (v5) at (2,-1) {};
\node [vertex] (v6) at (1,-1) {};
\node [redvertex] (v7) at (0,-1) {};
\node [vertex] (v8) at (-1,-1) {};

\draw[arc] (v1) -- (v2);
\draw[arc] (v2) -- (v3);
\draw[arc] (v3) -- (v4);
\draw[arc] (v4) -- (v5);
\draw[arc] (v5) -- (v6);
\draw[arc] (v6) -- (v7);
\draw[arc] (v7) -- (v8);
\draw[arc] (v8) -- (v1);
\draw[arc] (v7) -- (v2);
\draw[arc] (v3) -- (v6);
\draw[arc] (v7) -- (v1);
\draw[arc] (v3) -- (v7);
\draw[arc] (v5) -- (v3);
\end{scope}

\end{tikzpicture}}
\caption{Explicit optimal orientations for Lemma~\ref{lem:special_orientation}}
\end{figure}
Next we show Theorem~\ref{thm:ordiam_vs_diam}, which we restate for convenience.\diambound*
\begin{proof}
Let $G$ be a plane triangulation such that $G\neq K_4$, and let $G^*$ be the plane dual of $G$. It follows that $G^*$ is connected and $G^*$ is cubic as $G$ is a plane triangulation. Note that $G^*\neq K_4$ since $G\neq K_4$. Then Brooks' theorem gives a proper vertex $3$-coloring $\varphi: V(G^*)\to \{0,1,2\}.$ Equivalently, this is a proper $3$-coloring of the faces of $G$. 

We orient each edge of $G$ by the coloring $\varphi: F(G)\to \{0,1,2\}$. For each edge $e=xy$ of $G$, $e$ is contained in exactly two faces of $G$, say $f_e$ and $g_e$. Since $\varphi$ is proper, $\varphi(f_e)\neq \varphi(g_e)$. Orient every edge $e$ of $G$ so that its lower-colored incident face lies on its left, and we call the resulting orientation $D$. Note that the boundary of a face with color $0$ corresponds to a counterclockwise directed triangle in $D$ and the boundary of a face with color $2$ corresponds to a clockwise directed triangle in $D$. Since $\varphi$ is a proper $3$-coloring, we have that each edge of $G$ is incident with a face of color $0$ or $2$. This implies that every arc of $D$ is contained in a directed triangle.

For every $u,v\in V(G)$, let $P$ be a shortest $uv$-path of $G$ denoted by $P=x_0 x_1\ldots x_{\ell}$, where $x_0=u$ and $x_{\ell}=v$. We now find a directed $uv$-walk in $D$. If $x_ix_{i+1}$ is oriented $x_i \to x_{i+1}$, we use that arc. Otherwise, we have that there exists $z_i\in V(D)$ such that $x_i\to z_i\to x_{i+1}$ is a directed path of length two in $D$ as each arc of $D$ is contained in a directed triangle. Then $P$ gives a directed $uv$-walk of length at most $2\ell$. Hence, $d_D(u,v)\leq 2\ell =2d_G(u,v)$. Therefore, $\ordiam(G)\leq \diam(D)\leq 2\diam (G)$.
\end{proof}
Using the face coloring idea, we have the following lemma for near triangulations, which we  use to prove Conjecture~\ref{conj:main} for the $4$-connected case.
\begin{lemma}\label{lem:near_trian_inteorior}
 Let $G$ be be a near triangulation with outer cycle $C$. Then $G-E(C)$ admits an orientation $D$, such that 
 \begin{itemize}
 \item [(i)] for any $u_1,u_2\in V(G)\setminus V(C)$ with $u_1, u_2$ contained in a same component of $G-C$, we have that $d_D(u_1,u_2)\leq 2d_{G-C}(u_1, u_2)$, and 
\item [(ii)] for any $v\in V(G)\setminus V(C)$, we have that $\max\{d_D(v, V(C)), d_D(V(C), v)\}\leq 2 d_G(v, V(C)).$ 
\end{itemize}
\end{lemma}
\begin{proof}
  Let $G$ be a near triangulation with outer cycle $C$. Let $G^*_1$ be the weak dual of $G$ (the subgraph of the plane dual of $G$ induced on all faces except the outer face). Note that $G^*_1$ is connected as $G$ is $2$-connected. Since $G_1^*$ is not $K_4$ and has maximum degree at most three, Brooks' theorem gives a proper vertex $3$-coloring $\varphi_1: V(G_1^*)\to \{0,1,2\}$. Note that every edge in $G-E(C)$ is contained in two facial triangles bounding the interior faces of $G$, and so it is incident with a face of color $0$ or $2$. Hence, we use the same way as the proof of Theorem~\ref{thm:ordiam_vs_diam} to orient all the edges in $G-E(C)$, and then we let $D$ denote this orientation. Observe that for each edge $e\in E(G-C)$, all the edges in the two facial triangles containing $e$ are contained in $G-E(C)$. Similarly, we have each arc in $D-V(C)$ (from the edges in $G-C$) is contained in a directed triangle.

  Suppose $u_1, u_2\in V(G)\setminus V(C)$ are contained in the same component of $G-C$. Assume $P$ is a shortest $u_1u_2$-path in $G-C$. Since all edges of $P$ are contained in $G-C$, we know every arc from $E(P)$ is contained in a directed triangle of $D$.  Then it follows that $d_D(u_1, u_2)\leq 2d_{G-C}(u_1, u_2)$ as $P$ produces a $u_1u_2$-walk of length at most $2d_{G-C}(u_1, u_2)$ in $D$. This proves (i).

Next we show (ii). Let $v\in V(G)\setminus V(C)$ and let $Q=x_0 x_1 \ldots x_k$ for some integer $k$, where $x_0=v$ and $x_k\in V(C)$, denote a shortest path between $v$ and $V(C)$. Hence, $d_{G}(v, V(C))=k $ and we have that $x_i\in V(G)\setminus V(C)$ for $i\in [k-1]$ if $k\geq 2$.  Suppose $k=1$, i.e., $vx_1\in E(G)$ for some $x_1\in V(C)$. We know that the edge $vx_1$ is incident with a triangular face of color $0$ or $2$ in $\varphi_1$. Without loss of generality, we may assume $vx_1$ is incident with a triangular face of color $0$, and let $T=vx_1wv$ denote its boundary cycle. If all edges of $T$ are contained in $G-E(C)$, we know $v\to w \to x_1 \to v$ is a directed triangle of $D$. Hence, $d_D(v,x_1)\leq 2$ and $d_D(x_1, v)=1$. This implies that $\max\{d_D(v, V(C)), d_D(V(C), v)\}\leq 2 =2 d_{G}(v, V(C))$. We may now assume that $E(T)$ is not contained in $G-E(C)$. Then we have $x_1w\in E(C)$ as $v\in V(G)\setminus V(C)$. Hence, $w\in V(C)$ and we know that the edges $x_1v, vw$ are oriented as $x_1\to v, v\to w$ since the face with boundary $T=vx_1wv$ is colored with $0$ in $\varphi_1$. Hence, it follows that $d_D(v, V(C))= d_D(V(C), v)=1<2= 2d_{G}(v, V(C))$. Now we suppose $k\geq 2$. Observe that the distance between $x_{k-1}$ and $V(C)$ in $G$ is one. Thus, we have $\max \{d_D(x_{k-1}, V(C)), d_D(V(C), x_{k-1})\}\leq 2$. Note that $x_0Qx_{k-1}=vx_1\ldots x_{k-1}$ is a path of length $k-1$ in $G-C$, and it follows from (i) that $\max \{d_D(v, x_{k-1}), d_D(x_{k-1}, v)\}\leq 2(k-1)$. Therefore, 
\begin{align*}
d_D(v, V(C))\leq d_D(v, x_{k-1})+d_D(x_{k-1}, V(C))\leq 2(k-1)+2=2k=2d_G(v, V(C)).
\end{align*}
 Similarly, we have that $d_D(V(C), v)\leq2d_G(v, V(C))$, which completes the proof. 
\end{proof}
We also need the following useful observations and we leave the proofs to readers.
\begin{observation}
    \label{obs:goodH}
    Let $G$ be a graph and $H$ be a connected bridgeless spanning subgraph of $G$. Then \[\ordiam(G)\leq \ordiam(H).\]
\end{observation}
\begin{observation}
    \label{obs:goodH}
    Let $G$ be a graph and $H$ be a connected bridgeless subgraph of $G$. Suppose every vertex of $V(G)\setminus V(H)$ has at least two neighbors in $H$. Then \[\ordiam(G)\leq \ordiam(H)+2.\]
\end{observation}

\section{Planar triangulations with long diameter}\label{sec:long_diameter}
We first give some definitions. Let $G$ be a graph and $H\subseteq G$. An {\it $H$-bridge} of $G$ is a subgraph of $G$ induced by either an edge in $E(G)\setminus E(H)$ with both incident vertices in $V(H)$, or all edges in $G-E(H)$ with at least one incident vertex in a single component of $G-H$. For an  $H$-bridge $B$ of $G$, the vertices in $V(B\cap H)$ are the {\it attachments} of $B$ on $H$.

\longdiam*
\begin{proof}
 Suppose $G$ is a plane triangulation on $n$ vertices. We want to show $\ordiam(G)\leq \frac{n-\diam(G)}{2}+11$.  Note that $\ordiam(K_4)=3$, and so we may assume that $G\neq K_4$. Since $G$ is also a $2$-connected near triangulation, Theorem~\ref{thm:near_triangulation_diameter} implies that $\ordiam(G)\leq \frac{n+1}{2}$. Thus, if $\diam(G)\leq 21$ then $\ordiam(G)\leq \frac{n+1}{2}\leq \frac{n-\diam(G)}{2}+11$. Hence, we may assume that $\diam(G)\geq 22$.
 Let $u,v\in V(G)$ such that $d_G(u,v)=\diam(G)$. Then every $uv$-path in $G$ has at least $\diam(G)+1\geq 23$ vertices. Note that $G$ is $3$-connected as $G$ is a plane triangulation on at least four vertices. By Menger's theorem, there exist three internally vertex-disjoint $uv$-paths in $G$, say $P_0, P_1, P_2$. 
We may assume that all internal vertices of $P_0$ are contained in the interior of the cycle $C_0:=P_1\cup P_2$, and may assume $P_0$ is a shortest $uv$-path in $\overline{\Int(C_0)}-(V(C_0)\setminus \{u,v\})$. Let $G_0$ be the maximal $2$-connected graph containing $\overline{\Ext(C_0)}$ in $G-(V(P_0)\setminus \{u,v\})$. Note that $G_0$ is a $2$-connected near triangulation and $u,v$ are contained in its outer cycle. This implies that the outer cycle of $G_0$ contains two internally vertex-disjoint $uv$-paths, each of which is internally vertex-disjoint from $P_0$. Hence, we may assume $G_0=\overline{\Ext(C_0)}$ is a near triangulation with outer cycle $C_0=P_1\cup P_2$. We find a subgraph $H\subseteq G-E(P_0)$ such that $H$ contains $G-(V(P_0)\setminus \{u,v\})$ as a subgraph and $$\ordiam(H)\leq \frac{n+1-|V(P_0)|}{2}+9.$$ 
We first claim that the existence of $H$ implies that $\ordiam(G)\leq \frac{n-\diam(G)}{2}+11$.
Observe that every vertex in $P_0-\{u,v\}$ has at least two neighbors in $G-(V(P_0)\setminus \{u,v\})$ as $P_0$ is a shortest $uv$-path in $G-(V(G_0)\setminus \{u,v\})$, and every vertex in $V(G)\setminus V(H)$ is contained in $P_0-\{u,v\}$.  This implies that every vertex in $V(G)\setminus V(H)$ has at least two neighbors in $H$. It follows from Observation~\ref{obs:goodH} that $$\ordiam(G)\leq \ordiam(H)+2\leq \frac{n+1-|V(P_0)|}{2}+9+2\leq \frac{n-\diam(G)}{2}+11.$$

\begin{figure}[htbp]
\label{fig:bridge}
\centering
\resizebox{12cm}{!}{\begin{tikzpicture}[
    x=0.9cm,
    y=0.9cm,
    line/.style={draw=black, line width=0.8pt},
    vertex/.style={circle, fill=black, inner sep=2.2pt},
    every node/.style={font=\large}
]

% Main vertices on P_0
\coordinate (u)   at (0,0);
\coordinate (a1)  at (1.35,0);
\coordinate (a2)  at (2.65,0);
\coordinate (xj)  at (4.35,0);
\coordinate (xs)  at (5.45,0);
\coordinate (xt)  at (6.55,0);
\coordinate (x1)  at (7.2,0);
\coordinate (x2)  at (8.2,0);
\coordinate (xk)  at (9.75,0);
\coordinate (b0) at (10.2,0);
\coordinate (b1)  at (11,0);
\coordinate (b2)  at (12.15,0);
\coordinate (v)   at (14,0);

\node[vertex] at (6,0) {};
% Vertices below P_0
\coordinate (ell1) at (1.10,-2.25);
\coordinate (ell2) at (2.00,-3.00);
\coordinate (z1)   at (6.25,-1.35);
\coordinate (z2)   at (7.5, -1.35);
\coordinate (y1)   at (5.85,-4.35);
\coordinate (r1)   at (11.5,-2.85);

% Upper u-v path P_2
\draw[line]
    (u)
    .. controls (3.1,4.25) and (10.65,4.25) ..
    node[pos=.80, above=2pt] {$P_2$}
    (v);

% Horizontal path P_0, with two omitted portions
\draw[line]
    (u)--(a1)--(a2)--(2.95,0);

\node[
    fill=white,
    inner xsep=3pt,
    inner ysep=0pt
] at (3.48,0) {$\cdots$};

\draw[line]
    (3.98,0)--(xj)--(xs)--(xt)--(8.5,0);

\node[
    fill=white,
    inner xsep=3pt,
    inner ysep=0pt
] at (9,0) {$\cdots$};

\draw[line]
    (9.5,0)--(xk)--(b0)--(b1)--(b2)--(12.75,0);

\draw[line]
(13.25,0)--(v);
\node[
    fill=white,
    inner xsep=3pt,
    inner ysep=0pt
] at (13,0) {$\cdots$};
% Lower u-v path P_1
\draw[line]
    (u)
    .. controls (0.2,-1) and (0.7,-1.5) ..
    (ell1)--(ell2)
    .. controls (3.15,-3.85) and (4.65,-4.30) ..
    (y1)
    .. controls (7.85,-4.45) and (9.85,-3.75) ..
    (r1)
    .. controls (12.5,-2) and (13.5,-1) ..
    (v);

% Left triangulated part
\draw[line] (a1)--(ell1);
\draw[line] (a2)--(ell1);
\draw[line] (a2)--(ell2);
\draw[line] (xj)--(ell2);

% Central part
\draw[line] (xj)--(y1);
\draw[line] (xk)--(y1);

% z_1 is adjacent to x_{s_1} and x_{t_1}
\draw[line] (xs)--(z1)--(xt);

\draw[line] (x1)--(z2)--(x2);
% Curved edge between x_{s_1} and x_{t_1}, below P_0
\draw[line]
    (xs)
    to[out=-58,in=-122,looseness=1.35]
    (xt);
\draw[line]
    (a1)
    to[out=-58,in=-122,looseness=1.35]
    (a2);
\draw[line]
    (a2)
    to[out=-58,in=-122,looseness=1.35]
    (xj);
    \draw[line]
    (b0)
    to[out=-58,in=-122,looseness=1.35]
    (b1);
    \draw[line]
    (b1)
    to[out=-58,in=-122,looseness=1.35]
    (b2);
\draw[line] (x1)--(z2)--(x2);
\draw[line]
    (x1)
    to[out=-58,in=-122,looseness=1.35]
    (x2);
% Curved edge between z_1 and x_{k_1}
\draw[line]
    (z2)
    .. controls (8,-1.4) and (8.75,-1.00) ..
    (xk);
\draw[line]
    (z1)
    .. controls (7,-1.5) and (8,-1.4) ..
    (z2);
\draw[line]
    (xj)
    .. controls (4.9,-1.2) and (5.65,-1.10) ..
    (z1);
% Right triangulated part
\draw[line] (b0)--(r1);
\draw[line] (b1)--(r1);
\draw[line] (b2)--(r1);

% Draw all vertices last
\foreach \p in {
    u,a1,a2,xj,xs,xt,xk,b0,b1,b2,v,
    ell1,ell2,z1,z2,y1,r1,x1,x2}{
    \node[vertex] at (\p) {};
}

% Vertex labels
\node[left=5pt]  at (u) {$u=x_0$};
\node[right=5pt] at (v) {$v=x_{|V(P_0)|-1}$};

\node[above=5pt] at (xj) {$x_{j_1}$};
\node[above=5pt] at (xs) {$x_{s_1}$};
\node[above=5pt] at (xt) {$x_{t_1}$};
\node[above=5pt] at (xk) {$x_{k_1}$};

\node[below=5pt] at (z1) {$z_1$};
\node[below=5pt] at (y1) {$y_1$};

% Path and region labels
\node at (12.5,0.55) {$P_0$};
\node at (12.5,-2.5) {$P_1$};
\node at (7.5,-4.8) {$G_0$};

\end{tikzpicture}}
\caption{The paths $P_0, P_1, P_2$, and $G_0=\overline{\Ext(P_1\cup P_2)}$, and some $H_0$-bridges}
\end{figure}
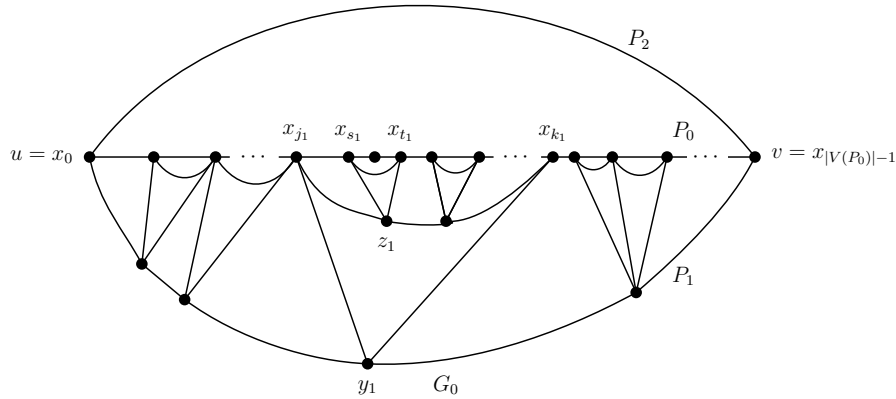
Now we find such $H$. Let $H_0=G_0\cup P_0$. Consider all $H_0$-bridges of $G$. Suppose $B$ is an $H_0$-bridge of $G$. If $B$ is an edge, we know that this edge has one end in $P_0-\{u,v\}$ and one end in $P_1-\{u,v\}$ or $P_2-\{u,v\}$ by the choices of $P_0$ and $G_0$. If $B$ is not an edge, then $B$ has at least three attachments; moreover, $B$ has exactly one attachment on $G_0$ by the maximality of $G_0$.  Observe that $G_0$ has at least $|V(P_1)|+|V(P_2)|-2\geq 2(\diam(G)+1)-2=2\diam(G)\geq 44$ vertices. Then it follows from Theorem~\ref{thm:near_triangulation_diameter} that $\ordiam(G_0)\leq {\frac{|V(G_0)|+1}{2}}$, and let $O_0$ be an optimal orientation of $G_0$. Hence, $\diam(O_0)=\ordiam(G_0)\leq {\frac{|V(G_0)|+1}{2}}$. Suppose all $H_0$-bridges are edges. Then $G-(V(P_0)\setminus \{u,v\})=G_0$, and we let $H=G_0$.  Note that $|V(H)|=|V(G_0)|=n-(|V(P_0)|-2)= n+2-|V(P_0)|$. It follows that $$\ordiam(H)=\ordiam(G_0)\leq{\frac{|V(G_0)|+1}{2}}= \frac{n+3-|V(P_0)|}{2}< \frac{n+1-|V(P_0)|}{2}+9.$$ Thus, we may now assume there exists an $H_0$-bridge that is not an edge.

\medskip

Let $B_1, \ldots, B_\ell$ for some integer $\ell\geq 1$ be all $H_0$-bridges of $G$ such that $B_i$ is not an edge for each $i\in [\ell]$. Hence $V(B_i)\setminus V(H_0)\neq \emptyset$. Recall that $B_i$ has a unique attachment on $C_0$, say $y_i$. Note that $y_i\in V(P_1)\setminus \{u,v\}$ or $y_i\in V(P_2)\setminus \{u,v\}$. Consider $B_1$. Without loss of generality, we may assume that $y_1$, the attachment of $B_1$ on $C_0$, is contained in  $V(P_1)\setminus \{u,v\}$. Let $P_0=x_0x_1\ldots x_{|V(P_0)|-1}$ such that $x_0=u, x_{|V(P_0)|-1}=v$. Let the integers $j_1, k_1$, respectively, with $1\leq j_1< k_1\leq |V(P_0)|-2$ denote the minimum index and the maximum index of the attachments of $B_1$ on $P_0$, respectively. Since $G$ is a plane triangulation, we have that $y_1$ is adjacent to $x_{j_1}$ and $x_{k_1}$. Observe that $y_1x_{j_1}, y_1x_{k_1}$ are not edges of $B_1$. We define a subgraph $G_1\subseteq B_1\cup x_{j_1}y_1x_{k_1}$ such that $G_1$ contains all vertices of $V(B_1)\setminus V(H_0)$. Note that $B_1$ contains no edge in $E(P_0) \cup \{y_1x_{j_1}, y_1x_{k_1}\}$ and $B_1$ has an $x_{j_1}x_{k_1}$-path, say $Q_1$, that does not contain the vertex $y_1$. Note that $Q_1$ has length at least the length of $x_{j_1}P_0x_{k_1}$.

Observe that all internal vertices of $x_{j_1}P_0 x_{k_1}$ are contained in the exterior of the cycle $C_1:=Q_1\cup x_{j_1}y_1 x_{k_1}$, and we may choose $Q_1\subseteq B_1$ such that $\overline {C_1}$ has the maximum possible number of vertices and edges. Suppose $\overline{C_1}$ contains all vertices in $V(B_1)\setminus V(H_0)$. The let $G_1:=\overline{C_1}$. We may now assume $(V(B_1)\setminus V(H_0))\setminus V(\overline{C_1})\neq \emptyset$. Let $C_1'=Q_1\cup x_{j_1} P_0 x_{k_1}$. Then $\overline{C_1'}$ is not outerplanar. Let $F_1, \ldots, F_{i_1}$ with some integer $i_1\geq 1$ be all $C_1'$-bridges of $\overline{C_1'}$ that are not a single edge.
By the choices of $Q_1$ and $P_0$, we know that $V(F_i)\cap \{x_{j_1}, x_{k_1}\}=\emptyset$, $F_i$ has exactly one attachment on $Q_1-\{x_{j_1}, x_{k_1}\}$, say $z_i$, and at least two attachments on $x_{j_1}P_0x_{k_1}-\{x_{j_1}, x_{k_1}\}$ for each $i\in [i_1]$. Observe that $z_i\in V(B_1)\setminus V(H_0)$.
Let the integers $s_i, t_i$ with $j_1+1\le s_i<t_i\le k_1-1$  denote the minimum index and the maximum index of the attachments of $F_i$ on $x_{j_1}P_0x_{k_1}$, respectively. Then $z_i$ is adjacent to $x_{s_i}$ and $x_{t_i}$, and by the choice of $P_0$, we have that $1\le t_i-s_i\le 2$.
By planarity, we may assume that $j_1+1\le s_1<t_1\le s_2< t_2\leq \ldots \leq s_{i_1}<t_{i_1}\le k_1-1$. Hence, we have that $k_1-j_1\geq i_1+2$ for $i_1\geq 1$. Consider $F_1$. Recall that $F_1$ has an attachment $z_1$ on $Q_1-\{x_{j_1}, x_{k_1}\}$, at least two attachments $x_{s_1}, x_{s_1+1},\ldots , x_{t_1}$ on $P_0$, and $z_1x_{s_1}, z_1x_{t_1}\in E(G)$. Observe that
$F_1$ has no edge in $\{z_1x_{s_1}, z_1x_{t_1}\}\cup E(C_1')$ with $E(C_1')\subseteq E(P_0)\cup E(Q_1)$.  Let $R_1$ be an $x_{s_1}x_{t_1}$-path of $F_1$ that does not contain the vertex $z_1$, and we may choose $R_1$ such that for the cycle $D_1=R_1\cup x_{s_1}z_1x_{t_1}$, the graph $\overline{D_1}$ has maximum possible number of vertices and edges. Since $1\le t_1-s_1\leq 2$, we obtain that $\overline{D_1}$ contains all vertices in $V(F_1)\setminus V(C_1')$. Define $D_2, \ldots, D_{i_1}$ by $F_2, \ldots, F_{i_1}$, respectively, in this way. Then we let $G_1:=\overline{C_1}\cup \overline {D_1}\cup \cdots \cup \overline{D_{i_1}}$. See Figure~\ref{fig:bridge} for the illustration of $G_1$. We repeat this process for each $B_i$ with $2\le i \le \ell$ if $\ell\geq 2$ and obtain $G_2, \ldots, G_{\ell}$. Let $H:=G_0\cup G_1\cup \cdots \cup G_\ell$.

\medskip
We find an orientation of $H$. Let $n_i:= |V(G_i)\setminus V(H_0)|=|V(B_i)\setminus V(H_0)|$ for each $i\in [\ell]$. We first obtain an orientation of $G_i$, say $O_i$, such that $\diam(O_i)\leq \frac{n_i}{2}+6$ and $\max \{d_{O_i}(y_i,x),d_{O_i}(x, y_i)\}\leq \frac{n_i}{2}+4$ for any $x\in V(G_i)$. It suffices to show that $G_1$ has such an orientation $O_1$. Observe that $\overline{C_1}$ is a $2$-connected near triangulation and $V(\overline{C_1})\cap V(H_0)=\{y_1, x_{j_1}, x_{k_1}\}$. It follows from Lemma~\ref{lem:special_orientation} that $\overline{C_1}$ has an orientation $M_1$ such that $\diam(M_1)=\ordiam(\overline{C_1})\leq \frac{|V(\overline{C_1})|}{2}+1\leq \frac{n_1+3}{2}+1=\frac{n_1}{2}+\frac{5}{2}$ and $\max \{d_{M_1}(y_1,x), d_{M_1}(x, y_1)\}\leq \frac{|V(\overline{C_1})|+1}{2}\leq \frac{n_1+3+1}{2}=\frac{n_1}{2}+2$ for any $x\in V(\overline{C_1})$.
Suppose $G_1=\overline{C_1}$. We let $O_1=M_1$, which is the desired orientation. We may now assume that $G_1=\overline{C_1}\cup \overline {D_1}\cup \cdots \cup \overline{D_{i_1}}$, where $i_1\geq 1$.  
Observe that $\overline{D_i}$ is a $2$-connected near triangulation on at least four vertices, $V(\overline{D_i})\cap V(H_0)=V(\overline{D_i})\cap V(P_0)=\{x_{s_i}, x_{t_i}\}$ and $V(\overline{D_i})\cap V(\overline{C_1})=\{z_i\}$ for each $i\in [i_1]$. Note that $|V(\overline{C_1})|+|V(\overline{D_1})|+\ldots+|V(\overline{D_{i_1}})|=n_1+3(i_1+1)$, and hence, $|V(\overline{C_1})|+|V(\overline{D_i})|\leq n_1+3(i_1+1)-4(i_1-1)=n_1-i_1+7$.
Lemma~\ref{lem:special_orientation} implies that each $\overline{D_i}$ has an orientation $N_i$ such that $\diam(N_i)=\ordiam(\overline{D_i})\leq \frac{|V(\overline{D_i})|}{2}+1$ and $\max \{d_{N_i}(z_i,x), d_{N_i}(x, z_i)\}\leq \frac{|V(\overline{D_i})|+1}{2}$ for any $x\in V(\overline{D_i})$.  Observe that for integers $i,j$ with $1\leq i < j \leq i_1$, if $V(\overline{D_i})\cap V(\overline{D_j})\neq \emptyset$ then $j=i+1$, and we have that $|E(\overline{D_i})\cap E(\overline{D_{i+1}})|\leq 1$. Note that the reverse of the orientation (the direction is reversed for each arc of the given orientation) we obtain from Lemma~\ref{lem:special_orientation} also satisfies the conclusion of Lemma~\ref{lem:special_orientation}. Hence,  we can combine the orientations $N_1, \ldots, N_{i_1}$ in order. (If $e$ is the unique edge in $E(\overline{D_i})\cap E(\overline{D_{i+1}})$ and $N_{i+1}$ orients $e$ in the opposite direction from $N_{i}$, then we replace $N_{i+1}$ by its reverse.) Since any $\overline{D_i}$ and $\overline{C_1}$ have no common edge, we can combine the orientation $N_1\cup \cdots \cup N_{i_1}$ with $M_1$. Let $O_1: =M_1\cup N_1\cup \cdots \cup N_{i_1}$. In particular, we have that 
\begin{align*}
\diam(M_1\cup N_i)&\leq \max \{\diam(M_1), \diam(N_i), \diam(M_1)+\frac{|V(\overline{D_i})|+1}{2}\}\\& 
\leq \frac{|V(\overline{C_1})|}{2}+1+\frac{|V(\overline{D_i})|+1}{2}=\frac{|V(\overline{C_1})|+|V(\overline{D_i})|+3}{2}\\&\leq \frac{n_1-i_1+7+3}{2}=\frac{n_1-i_1+10}{2}\leq \frac{n_1+9}{2}\\&< \frac{n_1}{2}+6.
\end{align*}
For any $x\in V(\overline{D_i})$, 
\begin{align*}
d_{M_1\cup N_i}(x, y_1)&\leq d_{N_i}(x, z_i)+d_{M_1}(z_i, y_1)\\&\leq \frac{|V(\overline{D_i})|+1}{2}+\frac{|V(\overline{C_1})|+1}{2}= \frac{|V(\overline{D_i})|+|V(\overline{C_1})|+2}{2}\\&\leq \frac{n_1-i_1+7+2}{2}\leq \frac{n_1}{2}+4.
\end{align*}
Similarly,  we also have that $d_{M_1\cup N_i}(y_1, x)\leq \frac{n_1}{2}+4$.  
We may now assume $i_1\geq 2$. Otherwise, $O_1=M_1\cup N_1$ is the desired orientation. It remains to show that for distinct $i,j\in [i_1]$, the distance between any vertex in $\overline{D_i}$ and any vertex in $\overline{D_j}$ is at most $\frac{n_1}{2}+6$ in $O_1$. Note that $|V(\overline{D_i})|+|V(\overline{D_j})|\leq n_1+3(i_1+1)-|V(\overline{C_1})|-4(i_1-2)=n_1-i_1+11-|V(\overline{C_1})|$. For any $v_i\in V(\overline{D_i})$ and $v_j\in V(\overline{D_j})$,
\begin{align*}
d_{O_1}(v_i, v_j)&\leq d_{N_i}(v_i, z_i)+d_{M_1}(z_i, z_j)+d_{N_j}(z_j, v_j)\\&
\leq \frac{|V(\overline{D_i})|+1}{2}+d_{M_1}(z_i, z_j)+\frac{|V(\overline{D_j})|+1}{2}\\&
\leq \frac{n_1-i_1+13-|V(\overline{C_1})|}{2}+d_{M_1}(z_i, z_j).
\end{align*}
Since $d_{M_1}(z_i,z_j)\leq \diam(M_1)\leq \frac{|V(\overline{C_1})|}{2}+1$, we obtain $d_{O_1}(v_i, v_j)\leq \frac{n_1-i_1+13}{2}+1\leq \frac{n_1}{2}+6$ when $i_1\geq 3$. Hence we may assume that $i_1=2$, and thus $G_1=\overline{C_1}\cup \overline{D_1}\cup \overline{D_2}$. If $\diam(M_1)=\ordiam(\overline{C_1})\leq \frac{|V(\overline{C_1})|+1}{2}$, we also obtain that $d_{O_1}(v_i, v_j)\leq \frac{n_1}{2}+6$. We may now assume that $\diam(M_1)=\ordiam(\overline{C_1})>\frac{|V(\overline{C_1})|+1}{2}$. It follows from Theorem~\ref{thm:near_triangulation_diameter} that $\overline{C_1}$ is a graph in $\{K_4^-, K_4, W_5, G_6^1, G_6^2, G_6^3, G_8^1\}$. Recall that $C_1=Q_1\cup x_{j_1}y_1 x_{k_1}$ where $Q_1$ is an $x_{j_1}x_{k_1}$-path in $B_1$ not containing $y_1$, $V(\overline{C_1})\cap V(P_0)=\{x_{j_1}, x_{k_1}\}$, and $k_1-j_1\geq i_1+2$ if $i_1\geq 1$. Hence, $k_1-j_1\geq 4$ as $i_1=2$ and so $x_{j_1}P_0x_{k_1}$ has length at least four. By the choice of $P_0$, we know that the distance between $x_{j_1}$ and $x_{k_1}$ in $\overline{C_1}-y_1$ is at least four. We can check that for any near triangulation $H'\in \{K_4^-, K_4, W_5, G_6^1, G_6^2, G_6^3, G_8^1\}$ with outer cycle $C$ and any vertex $w\in V(C)$, the two vertices in $N_C(w)$ have distance at most four in $H'-w$. This implies that $k_1-j_1=4$, and $\overline{C_1}\in \{G_6^2, G_8^1\}$. Recall that $V(\overline{D_1})\cap V(P_0)=\{x_{s_1}, x_{t_1}\}, V(\overline{D_2})\cap V(P_0)=\{x_{s_2}, x_{t_2}\}$ and $j_1+1\le s_1<t_1\leq s_2<t_2\leq k_1-1=j_1+3$. This implies that $t_1=s_2=j_1+2$ and $x_{t_1}=x_{s_2}$. Hence, $x_{t_1}\in V(\overline{D_1})\cap V(\overline{D_2})\neq \emptyset$. It follows that for $v_1\in V(\overline{D_1})$ and $v_2\in V(\overline{D_2})$, we have that 
\begin{align*}
\max \{d_{O_1}(v_1, v_2),d_{O_1}(v_2, v_1)\} &\leq \max\{d_{N_1}(v_1, x_{t_1})+d_{N_2}(x_{t_1}, v_2), d_{N_2}(v_2, x_{t_1})+d_{N_1}(x_{t_1}, v_1)\} \\&\leq \frac{|V(\overline{D_1})|}{2}+1 + \frac{|V(\overline{D_2})|}{2}+1=\frac{n_1+9-|V(\overline{C_1})|}{2}+2\\&\leq \frac{n_1+9-6}{2}+2=\frac{n_1}{2}+\frac{7}{2}<\frac{n_1}{2}+6.
\end{align*}Therefore, for each $i\in [\ell]$, there exists an orientation $O_i$ of $G_i$ such that $\diam(O_i)\leq \frac{n_i}{2}+6$ and $\max \{ d_{O_i} (y_i, x), d_{O_i} ( x,y_i)\}\leq \frac{n_i}{2}+4$ for any $x\in V(G_i)$. Now we are ready to give an orientation $O_H$ of $H=G_0\cup G_1\cup \cdots \cup G_{\ell}$ such that \[\diam(O_H)\leq \frac{n+1-|V(P_0)|}{2}+9.\] 

Recall that $G_0$ has an orientation $O_0$ such that $\diam(O_0)=\ordiam(G_0)\leq \frac{|V(G_0)|+1}{2}$. Similarly, by planarity, the structures of $G_0, G_1, \ldots, G_{\ell}$ allow us to combine the orientations $O_0, O_1, \ldots, O_{\ell}$. Let $O_H:=O_0\cup O_1 \cup \cdots \cup O_{\ell}$. Observe that $|V(G_0)|+n_1+\cdots + n_{\ell}+|V(P_0)|-2=n$, and $n_i\geq 1$ for each $i\in [\ell]$. We know that 
\begin{align*}
\diam(O_0)=\ordiam(G_0)&\leq \frac{|V(G_0)|+1}{2}\leq \frac{n+2-\ell-|V(P_0)|+1}{2}\\&\leq  \frac{n+2-|V(P_0)|}{2}<\frac{n+1-|V(P_0)|}{2}+9,
\end{align*}
and 
\begin{align*}
\diam(O_i)&\leq \frac{n_i}{2}+6\leq \frac{n+2-|V(G_0)|-(\ell-1)-|V(P_0)|}{2}+6\\&\leq \frac{n+2-|V(G_0)|-|V(P_0)|}{2}+6<\frac{n+1-|V(P_0)|}{2}+9.
\end{align*}
For $v_i\in V(G_i)$ and $v_0\in V(G_0)$, we have that 
\begin{align*}
d_{O_H}(v_i, v_0)&\leq d_{O_i}(v_i, y_i)+d_{O_0}(y_i, v_0)\\&\leq \frac{n_i}{2}+4+\frac{|V(G_0)|+1}{2}=\frac{n_i+|V(G_0)|+1}{2}+4\\&\leq \frac{n+2-(\ell-1)-|V(P_0)|+1}{2}+4\\&\leq \frac{n+1-|V(P_0)|}{2}+5<\frac{n+1-|V(P_0)|}{2}+9.
\end{align*}
Similarly, we have that $d_{O_H}(v_0, v_i)\leq \frac{n+1-|V(P_0)|}{2}+5<\frac{n+1-|V(P_0)|}{2}+9$. Hence, we may assume that $\ell\geq 2$, and it remains to show for any distinct $i,j\in [\ell]$ and any $v_i\in V(G_i), v_j\in V(G_j)$, 
\begin{align*}d_{O_H}(v_i, v_j)&\leq d_{O_i}(v_i, y_i)+d_{O_0}(y_i, y_j)+d_{O_j}(y_j, v_j)\\&\leq \frac{n_i}{2}+4+\frac{|V(G_0)|+1}{2}+\frac{n_j}{2}+4=\frac{n_i+n_j+|V(G_0)|+1}{2}+8\\&\leq \frac{n+2-(\ell-2)-|V(P_0)|+1}{2}+8\leq \frac{n+1-|V(P_0)|}{2}+9.
\end{align*}
This completes the proof.
\end{proof}

\section{Planar triangulations with high connectivities} \label{sec:4-conn}
\highconn*
\begin{proof}
Let $G$ be an $n$-vertex $4$-connected plane triangulation. We want to prove that $\ordiam(G)\leq \frac{n-2\diam(G)}{2}+\frac{17}{2}$. We may assume $\diam(G)\geq 9$ by Theorem~\ref{thm:near_triangulation_diameter}. Let $u,v\in V(G)$ such that $d_G(u,v)=\diam(G)$. Since $G$ is $4$-connected, there exist four internally vertex-disjoint $uv$-paths in $G$, say $Q_0, P_0, P_1, P_2$. We may assume that all internal vertices of $Q_0, P_0$ are contained in the interior of the cycle $C_0:=P_1\cup P_2$, and hence, all the internal vertices of $P_1,P_2$ are contained in the exterior of the cycle $C:=Q_0\cup P_0$. We may choose $Q_0, P_0, P_1, P_2$ such that the interior of $C=Q_0\cup P_0$ has maximal vertices and edges, and the exterior of $C_0=P_1\cup P_2$ has maximal vertices and edges. Since $G$ is a plane triangulation, it follows that $V(G)=V(\overline{\Int(C)})\cup V(\overline{\Ext(C_0)})$. Moreover, by the choices of $Q_0, P_0$, we know that every vertex in $C=Q_0\cup P_0$ has a neighbor in $C_0=P_1\cup P_2$.%, and every vertex in $C_0=P_1\cup P_2$ has a neighbor in $C=P_1\cup P_2$. 
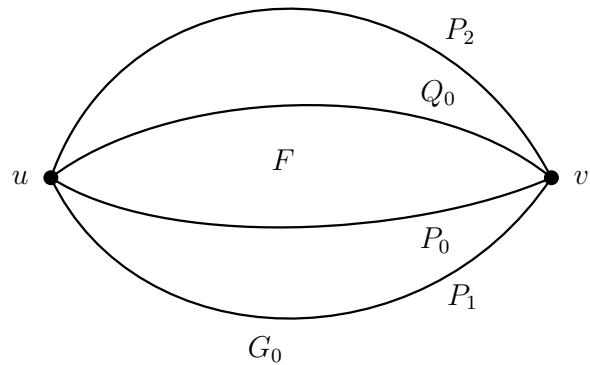
\begin{figure}[htbp]
\centering
\resizebox{8cm}{!}{\begin{tikzpicture}[
    x=1cm,
    y=1cm,
    path/.style={draw=black, line width=0.9pt},
    vertex/.style={circle, fill=black, inner sep=2.1pt},
    every node/.style={font=\large}
]

% Endvertices
\coordinate (u) at (0,0);
\coordinate (v) at (7,0);

% The four u--v paths
\draw[path]
    (u) .. controls (1.05,3) and (5.25,3.3) ..
    node[pos=.76, above right=-1pt] {$P_2$}
    (v);

\draw[path]
    (u) .. controls (1.55,1.2) and (5.15,1.5) ..
    node[pos=.76, above=2pt] {$Q_0$}
    (v);

\draw[path]
    (u) .. controls (1.55,-1) and (5.05,-0.85) ..
    node[pos=.76, below=2pt] {$P_0$}
    (v);

\draw[path]
    (u) .. controls (1.20,-2.55) and (5.20,-2.70) ..
    node[pos=.77, below right=-2pt] {$P_1$}
    (v);

% Endvertices and labels
\node[vertex] at (u) {};
\node[vertex] at (v) {};
\node[left=5pt]  at (u) {$u$};
\node[right=5pt] at (v) {$v$};

% Region label
\node at (3.25,0.25) {$F$};
\node at (3,-2.4) {$G_0$};
\end{tikzpicture}}
\caption{The paths $Q_0, P_0, P_1, P_2,$ and $F=\overline{\Int(Q_0\cup P_0)}, G_0=\overline{\Ext(P_1\cup P_2)}$}
\end{figure}

Let $F:=\overline{\Int(C)}=\overline{C}$ and $G_0:=\overline{\Ext(C_0)}$. Note that $F$ is a near triangulation with outer cycle $C=Q_0\cup P_0$, $G_0$ is a near triangulation with outer cycle $C_0=P_1\cup P_2$, and $V(F)\cap V(G_0)=\{u,v\}$. Hence, $|V(F)|+|V(G_0)|=n+2$. Since $G_0$ has at least $|V(P_1)|+|V(P_2)|-2\geq 2\diam(G)\geq 18$ vertices, it follows from Theorem~\ref{thm:near_triangulation_diameter} that $\ordiam(G_0)\leq \frac{|V(G_0)|+1}{2}$. Let $O_0$ be an optimal orientation of $G_0$, and so $\diam(O_0)=\ordiam(G_0)\leq \frac{n-|V(F)|+3}{2}$. Let $H:=\overline{\Ext(C)}$ and then $H$ is a near triangulation with outer cycle $C$. We know that $G_0$ is a subgraph of $H$. We claim that $\ordiam(H)\leq \ordiam(G_0)+4$.

Recall that $O_0$ is an optimal orientation of $G_0$. It suffices to find an orientation $O_H$ of a spanning subgraph of $H$ such that $\diam(O_H)\leq \diam(O_0)+4$. Note that $G_0\subseteq H$ and $V(H)\setminus V(G_0)=V(Q_0)\cup V(P_0)\setminus \{u,v\}$. To obtain $O_H$, first we orient each edge in $E(G_0)\subseteq E(H)$ by $O_0$, and then for each $x\in V(Q_0)\cup V(P_0)\setminus \{u,v\}$, we orient some edges incident with $x$ so that $\max\{d_{O_H}(x, V(G_0)), d_{O_H}( V(G_0),x)\}\leq 2$. Define $O_H^{(0)}$ such that the vertex set $V(O_H^{(0)}):=V(H)$ and the arc set $A(O_H^{(0)}):=A(O_0)$.    By the choices of $Q_0, P_0$, we know every vertex in $Q_0\cup P_0-\{u,v\}$ has at least one neighbor in $G_0$. We first consider the vertices in $P_0-\{u,v\}$. Suppose $P_0=x_0x_1\ldots x_{|V(P_0)|-1}$, where $x_0=u, x_{|V(P_0)|-1}=v$. We know that  $x_j$ is adjacent to a vertex in $G_0$, say $a_j$, for each $j\in [|V(P_0)|-2]$. Consider $x_1$. Note that $x_1x_0, x_1a_1$ are two edges in $E(H)\setminus E(G_0)$ and they are not oriented. We orient those two edges with one toward $x_1$ and one away from $x_1$, for example, $x_1 \to x_0, a_1\to x_1$, and we let $ O_H^{(1)}:=O_H^{(0)}\cup \{(x_1, x_0), (a_1, x_1)\}$. Hence, $\max \{d_{O_H^{(1)}}(x_1, V(G_0)), d_{O_H^{(1)}}( V(G_0), x_1)\}\leq 1$.
For each $j\in \{2, \ldots, |V(P_0)|-2\}$, we build $O_H^{(j)}$ according to the following process. If $x_j$ has two neighbors in $G_0$, say $a_j$ and $a_j'$, we let $O_H^{(j)}:=O_H^{(j-1)}\cup \{ (x_j, a_j), (a_j', x_j)\}$ and so $\max \{d_{O_H^{(j)}}(x_j, V(G_0)), d_{O_H^{(j)}}( V(G_0), x_j)\}\leq 1$. Now we suppose $a_j$ is the unique neighbor of $x_j$ in $G_0$. It follows that $a_j$ is also adjacent to $x_{j-1}$ as $G$ is a plane triangulation. If the edge $x_{j-1}a_j$ is oriented in $O_H^{(j-1)}$ then we orient the edges $x_{j-1}x_j, x_j a_j $ such that $x_{j-1}, x_j, a_j$ is contained in a directed triangle. Otherwise, we just make the triangle $x_{j-1}x_j a_j x_{j-1}$ a directed triangle. Let $O_H^{(j)}$ denote the new orientation obtained by adding the arcs from the triangle $x_{j-1}x_j a_j x_{j-1}$ to $O_H^{(j-1)}$. Observe that $\max\{d_{O_H^{(j)}}(x_j, a_j),d_{O_H^{(j)}}(a_j, x_j) \}\leq 2$, and so $\max\{d_{O_H^{(j)}}(x_j, V(G_0)),d_{O_H^{(j)}}(V(G_0), x_j) \}\leq 2$. Note that $O_H^{(|V(P_0)|-2)}$ contains all vertices of $P_0$, and in this orientation, every vertex $x$ in $P_0$ has $\max\{d_{O_H}(x, V(G_0)), d_{O_H}( V(G_0),x)\}\leq 2$. We repeat the same process to the vertices in $Q_0$, and get an orientation $O_H$ such that for any vertex $x\in V(P_0\cup Q_0)$, we have $\max \{d_{O_H} (x, V(G_0)), d_{O_H}(V(G_0), x)\}\leq 2$. Recall that for any $v_1, v_2\in V(G_0)$, we have that $d_{O_H} (v_1, v_2)\leq d_{O_0}(v_1, v_2)\leq \diam(O_0)$. It follows that \[\ordiam(H)\leq \diam(O_H)\leq \diam(O_0)+4=\ordiam(G_0)+4\leq \frac{n-|V(F)|+3}{2}+4.\]

Now we orient the edges in $\overline{C}-E(C)=F-E(C)$. Lemma~\ref{lem:near_trian_inteorior} gives an orientation $D$ of $F-E(C)$ such that for any $u_1, u_2\in V(F)\setminus V(C)$ with $u_1, u_2$ contained in a same component of $F-C$, we have $d_{D}(u_1, u_2)\leq 2d_{F-C}(u_1, u_2)$; and for any $w\in V({F})\setminus V(C)$, we have that 
\[\max\{d_{D}(w, V(C)), d_{D}(V(C), w)\}\leq 2 d_{{F}}(w, V(C)).\] 
Note that $E(H)\cap E(F-E(C))=\emptyset$ and $G=H\cup (F-E(C))$. We let $O:=O_H\cup D$, which is an orientation of a spanning subgraph of $G$. We only need to show that $$\diam(O)\leq \frac{n-(|V(Q_0)|+|V(P_0)|-2)}{2}+\frac{17}{2},$$ and hence, \[\ordiam(G)\leq\diam(O)\leq \frac{n-(|V(Q_0)|+|V(P_0)|-2)}{2}+\frac{17}{2}\leq \frac{n-2\diam(G)}{2}+\frac{17}{2}. \]
We may assume that $V({F})\setminus V(C)\neq \emptyset$. Otherwise, \[\diam(O)\leq \diam(O_H)\leq \frac{n+3-|V({F})|}{2}+4 \leq \frac{n+3-(|V(Q_0)|+|V(P_0)|-2|)}{2}+4.\]
Then let $B_1, \ldots, B_s$, for some integer $s\geq 1$, denote all the components of $\overline{C}-C={F}-C$, and let $n_i:=|V(B_i)|$ for each $i\in [s]$. Note that $n_1+\ldots +n_s+|V(Q_0)|+|V(P_0)|-2=|V({F})|$. We first claim that for each vertex $w_i\in V(B_i)$, we have that $d_{{F}}(w_i, V(C))\leq \frac{n_i+3}{4}$. Since $G$ is $4$-connected, we know there exist four paths from $w_i$ to $V(C)$ such that each path ends at a distinct vertex in $C$ and the paths share only the starting vertex $w_i$ with all the internal vertices contained in $B_i$. Then $d_{{F}}(w_i, V(C))\leq \frac{n_i+3}{4}$ follows from the existence of such paths. Thus, for any vertex $w_i\in V(B_i)$, we have that\[\max\{d_{D}(w_i, V(C)), d_{D}(V(C), w_i)\}\leq 2 d_{{F}}(w_i, V(C))\leq \frac{n_i+3}{2},\] 
  and then for any $w' \in V(H)$, we have that 
  \begin{align*}d_O(w_i, w')&\leq d_{D}(w_i, V(C))+\diam(O_H)\\&\leq \frac{n_i+3}{2} +\frac{n+3-|V({F})|}{2}+4\\&\leq  \frac{n-(|V({F})|-n_i)}{2}+7\\&\leq \frac{n-(|V(Q_0)|+|V(P_0)|-2)}{2}+7.
\end{align*}
  When $s\geq 2$, for distinct $i,j\in [S]$ and $w_i\in V(B_i), w_j\in V(B_j)$, we have that 
  \begin{align*}d_{O}(w_i, w_j)&\leq d_{D}(w_i, V(C))+\diam(O_H)+d_{D}(V(C), w_j)\\&
  \leq \frac{n_i+3}{2} +\frac{n+3-|V({F})|}{2}+4+ \frac{n_j+3}{2}\\&\leq \frac{n-(|V({F})|-n_i-n_j)}{2}+\frac{17}{2}
  \\& \leq \frac{n-(|V(Q_0)|+|V(P_0)|-2)}{2}+\frac{17}{2}.
  \end{align*}
Now it remains to show that for any $u_i, w_i\in V(B_i)$,  the distance from $u_i$ to $w_i$ in $O$ is at most $\frac{n-(|V(Q_0)|+|V(P_0)|-2)}{2}+\frac{17}{2}$.
Note that for any $u_iw_i$-path $P$ in $G$, either $P\subseteq B_i$ or $P$ contains two subpaths with one from $u_i$ to $V(F)$, one from $w_i$ to $V(C)$ and both having internal vertices contained in $B_i$. We know that $G$ has four internally vertex-disjoint $u_iw_i$-paths by $4$-connectedness of $G$. The existence of such four $u_iw_i$-paths implies that either $d_{F-C}(u_i, w_i)=d_{B_i}(u_i, w_i)\leq \frac{n_i+2}{4}$ or $d_{{F}}(u_i, V(C)) +d_{{F}}(w_i, V(C))\leq \frac{n_i+6}{4}$. 
The first case gives that 
\begin{align*}d_O(u_i, w_i)&\leq d_{D}(u_i, w_i)\leq 2d_{F-C}(u_i, w_i)\leq \frac{n_i+2}{2}\\&\leq \frac{n-(|V(Q_0)|+|V(P_0)|+|V(P_1)|+|V(P_2)|-6)+2}{2}
\\& \leq \frac{n-(|V(Q_0)|+|V(P_0)|+10+10-6)+2}{2}
\\& < \frac{n-(|V(Q_0)|+|V(P_0)|-2)}{2}+\frac{17}{2}
\end{align*}
as each $uv$-path has at least $\diam(G)+1\geq 10$ vertices, and the second case implies that 
\begin{align*}d_{O}(u_i, w_i)&\leq d_{D}(u_i, V(C))+\diam(O_H)+d_{D}(V(C), w_i)\\&\leq 2d_{{F}}(u_i, V(C))+ \diam(O_H)+2d_{{F}}(w_i, V(C))\\&\leq \frac{n_i+6}{2}+\frac{n+3-|V({F})|}{2}+4
\\& \leq \frac{n-(|V({F})|-n_i)}{2}+\frac{17}{2}\\&
\leq \frac{n-(|V(Q_0)|+|V(P_0)|-2)}{2}+\frac{17}{2}.
\end{align*}
This completes the proof.
\end{proof}
\begin{remark}
The argument for Theorem~\ref{main:4-conn} also yields a stronger
diameter-sensitive bound in the $5$-connected case as $5$-connectedness gives better distance bounds among the vertices in $F$ and $V(C)$. We briefly indicate
the necessary modifications.

 In the $5$-connected case, Menger's theorem allows us to choose $C$ so that $F=\overline{\Int(C)}$ has three internally vertex-disjoint $uv$-paths, 
 while $H=\overline{\Ext(C)}$ has $\ordiam(H)\leq \frac{n-|V(F)|+3}{2}+4$. Let $B$ be a component of $V(F)\setminus V(C)$, and let $n_B:=|V(B)|$. The $5$-connectivity of $G$ gives  $$d_F(w, V(C))\leq \frac{n_B}{5}+O(1)$$
for any $w\in V(B)$. For $w,w'\in V(B)$, either $d_{F-C}(w, w')=d_B(w, w')=\frac{n_B}{5}+O(1)$ or $$d_F(w, V(C))+d_F(w', V(C))\leq \frac{n_B}{5}+O(1).$$ Then with an orientation of $F-E(C)$ from Lemma~\ref{lem:near_trian_inteorior}, we have that
\begin{align*}\ordiam(G) &\leq \ordiam(H)+\frac{2}{5}|V(F)\setminus V(C)|+O(1)\\&
\leq \frac{n-|V(F)|}{2}+\frac{2}{5}(|V(F)|- |V(C)|)+O(1)
\\& \leq  \frac{n-(|V(C)|+\frac{1}{5}(|V(F)|-|V(C)|))}{2}+O(1)\\&\leq \frac{n-\frac{11}{5}\diam(G)}{2}+O(1).
\end{align*}
As the remainder of the argument is identical to the proof of
Theorem~\ref{main:4-conn}, we omit the details.

Combining with Theorem~\ref{thm:ordiam_vs_diam}, we obtain that every $n$-vertex $5$-connected planar triangulation $G$ has $\ordiam(G)\leq \frac{10}{31}n+O(1)$.
\end{remark}

\section*{Acknowledgments}
The author is grateful to Xingxing Yu and Zhiyu Wang for helpful discussions on this problem. The author used ChatGPT (OpenAI) in developing the face-coloring argument leading to Theorem~\ref{thm:ordiam_vs_diam} and language polishing.

\end{document}